\documentclass[aps,prd,groupedaddress,11pt,tightenlines,nofootinbib]{revtex4}
\usepackage{amsmath,amssymb,amsthm}

\def\xnewpage{} 

\newtheorem{Theorem}{Theorem}
\newtheorem{Cor}{Corollary}
\newtheorem{Prop}{Proposition}

\def\R{\mathbb{R}}

\def\d{\partial}
\def\w#1{\widetilde{#1}}
\newcommand{\fw}{\ensuremath{\;|\;}}
\newcommand{\E}{\ensuremath{\mathcal{E}}}
\newcommand{\T}{\ensuremath{\mathcal{T}}}
\newcommand{\F}{Q}
\DeclareMathOperator{\Div}{div}
\DeclareMathOperator{\Flux}{Flux}
\DeclareMathOperator{\grad}{grad}

\begin{document}

\title[Decay estimates for nonlinear wave equations]
{Large data pointwise decay for defocusing semilinear wave equations}

\author{Roger Bieli}
\author{Nikodem Szpak}
\address{Max-Planck-Institut f\"{u}r Gravitationsphysik,
Albert-Einstein-Institut\\ Am M\"uhlenberg~1\\ 14476 Golm, Germany}
\date{\today}

\begin{abstract}
We generalize the pointwise decay estimates for large data solutions of the defocusing semilinear wave equations which we obtained earlier under restriction to spherical symmetry. Without the symmetry the conformal transformation we use provides only a weak decay. This can, however, in the next step be improved to the optimal decay estimate suggested by the radial case and small data results.
This is the first result of that kind.
\end{abstract}

\maketitle

\section{Introduction}

This Article builds up on our previous result \cite{BSzpak-Defocus} where we considered  wave equations with the defocusing nonlinearity
\begin{equation} \label{eq:wave-p}
  \Box \phi = -|\phi|^{p-1}\phi, \qquad\qquad \Box \equiv \d_t^2  - \Delta, 
\end{equation}
in $n=3$ spatial dimensions restricted to spherical symmetry. Here, by generalizing the technique we can remove the symmetry condition and prove the same decay result
\begin{equation} \label{eq:main-estimate}
  |\phi(t,x)| \leq \frac{C}{(1+t+|x|)(1+t-|x|)^{p-2}}
\end{equation}
with some constant $C$ depending only on initial data and the power $3\leq p<5$. For the initial data $(\phi_0,\phi_1)$ we assume the regularity
\begin{align} \label{eq:initdata}
  (\phi, \d_t\phi)_{t=1} &= (\phi_0,\phi_1) \in C^3(\R^3) \times C^2(\R^3),
\end{align}
implying existence of classical solutions $\phi\in C^2(\R\times\R^3)$, and compact support\footnote{Symmetries of the equation \eqref{eq:wave-p} allow for mapping any compactly supported initial data onto the region $|x|\in[0,\alpha[$ at $t=1$.} in $|x|<\alpha:=\frac1{2}$ which, among others, guarantees that the positive definite energy
\begin{equation} \label{eq:gerg}
  E=\int_{\R^3} \left(\frac{1}{2}|\d_t \phi|^2 + \frac{1}{2}|\nabla
  \phi|^2+ \frac{1}{p+1}|\phi|^{p+1}\right)\,d^3 x
\end{equation}
is finite but not necessarily small.

Under these conditions J{\"o}rgens \cite{Joergens} has shown global existence, Pecher \cite{Pecher} uniform boundedness and Strauss \cite{Strauss_SemilinDecay} weak uniform decay  $1/t^{1-\epsilon}$. More references to similar results can be found in our previous publication on that equation \cite{BSzpak-Defocus}.
We are not aware of any stronger pointwise estimates existing in the literature.
While under spherical symmetry the decay rates were well motivated by the numerical analysis \cite{PB+TCh-private} and an asymptotic result for small initial data \cite{NS-PB_Tails} stating that for large $t$ and fixed $r$ the solution behaves like
\begin{equation*}
  \phi(t,r) = \frac{C}{t^{p-1}} + {\mathcal{O}}(t^{-p}),
\end{equation*}
such strong statements are lacking in the non-spherical case. We are only guided by our spherical result for big data \cite{BSzpak-Defocus} and a series of results giving pointwise decay estimates for small initial data \cite{Asakura, Strauss-T, NS-Tails}. In this context the estimate \eqref{eq:main-estimate} seems optimal.

Here, we combine the technique of conformal compactification developed by Choquet-Bruhat, Christodoulou and others in \cite{ChoqBruh-P-Segal,Christodoulou-Quasilin,Baez-Segal-Zhou}, the boundedness result for big data of Pecher \cite{Pecher} and some pointwise estimates developed by Asakura \cite{Asakura} and one of us \cite{NS-DecayLemma}.

Our method has a few steps. First we perform a conformal compactification of the space-time. We map the double-null coordinates $u=t+r$, $v=t-r$ where $r=|x|$ and the angular coordinates $\theta, \varphi$ according to
\begin{align}\label{eq:uv-map}
  \w{u} &:= -\frac1{u},& \w{v} &:= -\frac1{v},& \w{\theta} &:= \theta,& \w{\varphi} &:= \varphi.
\end{align}
The conformal factor $\Omega := {\w{r}}/{r}$ multiplies the transformed
solution $ \w{\phi}(\w{u},\w{v}) := \Omega^{-1} \phi(u,v) $ which satisfies
the transformed wave equation
\begin{equation} \label{eq:wave-mapped}
 \w{\Box} \w{\phi} +  (\w{u}\w{v})^{p-3}\w{\phi} |\w{\phi}|^{p-1} = 0
\end{equation}
in a precompact region of spacetime. There, we are able to show
boundedness of some pseudo-energy flux which we further use to show the
uniform boundedness of $|\w{\phi}(\w{u},\w{v})|\leq \w{C}$. Finally,
the inverse transformation gives us the weak pointwise estimate
\begin{equation} \label{eq:weak-decay}
  |\phi(u,v)| \leq \w{C}\cdot \Omega = \frac{\w{C}}{uv}.
\end{equation}
In the radial case (cf. \cite{BSzpak-Defocus}) there was a family of conformal transformations available, being the mappings of the pair $(u,v)$, which could be adjusted to the power of the nonlinearity $p$ in such a way that the conformal factor $\Omega$ already delivered the desired optimal pointwise decay.
Without the spherical symmetry these transformations are no more conformal except the one which was related to $p=3$ and is now being used. Hence, the rate of decay in \eqref{eq:weak-decay} corresponds to that of the radial problem with $p=3$.

In order to improve it (for $p>3$ the decay is faster), we need a second step in which we make use of an integral representation of \eqref{eq:wave-p}, known as the Duhamel formula, and its estimates, briefly
\begin{equation}
  |\phi(t,x)| \leq \Box^{-1} |\phi|^p \lesssim \Box^{-1} \frac1{(uv)^p} \lesssim \frac1{(1+t+|x|)(1+t-|x|)^{p-2}}.
\end{equation}

\xnewpage
\section{Conformal transformation} \label{sec:ct}

On $\R \times \R^3$, let $\T^+:=\{(t,x) \fw 0\leq |x| < t\}$ and
$\T^-:=\{(t,x) \fw 0\leq |x| < -t\}$ denote the interior of the forward
and backward lightcone of the origin respectively, where the coordinates
$t:\R \times \R^3 \to \R$ and $x:\R \times \R^3 \to \R^3$ are the canonical
projections. Consider the map
\begin{equation} \label{eq:confmap}
 \Phi : \T^+ \to \T^-,\quad (t,x) \mapsto \frac{(-t,x)}{t^2-x^2}.
\end{equation}
It is an analytic bijection with analytic inverse \[ \Phi^{-1} : \T^-
\to \T^+,\quad (t,x) \mapsto \frac{(-t,x)}{t^2-x^2}.\]

If $\eta := dt^2 -
\delta_{ij}\,dx^i dx^j$ denotes the Minkowski metric on $\R \times \R^3$,
then its pullback $\Phi^* \eta$ by $\Phi$ satisfies \[ \Phi^* \eta =
\Omega^2 \eta \] on $\T^+$, where \[ \Omega := \frac1{t^2-x^2} \] as $\Phi^*
dt = (2t^2 \Omega -1)\Omega \,dt - 2tx_i \Omega^2 \,dx^i$ and $\Phi^* dx^i =
-2tx^i \Omega^2 \,dt + (\delta^i_j + 2x^i x_j \Omega)\Omega \,dx^j$. Hence,
the mapping $\Phi$ is also conformal. The conformal factor $\Omega$ is
analytic and has the property that $\Phi^* |x| = \Omega |x|$.

From the calculations in the spherically symmetric case, more precisely
from Equation~\eqref{eq:BS-Defocus:ct} in \cite{BSzpak-Defocus}, it easily
follows that if $h \in C^2(\T^-)$ is a twice continuously differentiable
function on $\T^-$, its pullback $\Phi^* h \in C^2(\T^+)$ and satisfies
\begin{equation} \label{eq:ct}
 \Box( \Omega \Phi^* h ) = \Omega^3 \Phi^* \Box h.
\end{equation}
In particular, let $\phi \in C^2(\T^+)$ be a classical solution of the
semilinear wave equation
\begin{equation} \label{eq:wav}
 \Box \phi + \phi |\phi|^{p-1} = 0
\end{equation}
for some $p>2$, then its conformal transformation $\psi:=\Phi_*(\Omega^{-1}
\phi) \in C^2(\T^-)$ satisfies
\begin{equation} \label{eq:ctwav}
 \Box \psi + (\Phi_* \Omega)^{p-3} \psi |\psi|^{p-1} = 0,
\end{equation}
where $\Phi_*$ denotes the push-forward by the diffeomorphism
$\Phi$. On the other hand, it follows directly from \eqref{eq:ct}
that if $\psi \in C^2(\T^-)$ is a classical solution of the transformed
equation~\eqref{eq:ctwav} on $\T^-$, the function $\phi:= \Omega \Phi^*
\psi \in C^2(\T^+)$ solves the original equation~\eqref{eq:wav} on $\T^+$.
Note that for $p\geq 3$, the non-linearity in Equation~\eqref{eq:ctwav} is
regular at the origin. In the case $p=3$ the functions $\phi$ and $\psi$
solve the \emph{same} Equation~\eqref{eq:wav}, which is in this case called
conformally invariant.

The method of improving a boundedness result for $\psi$ into a decay
estimate for $\phi$ by means of the conformal transformation described
above is directly related to the Morawetz vector field \cite{Morawetz62}
\begin{equation} \label{eq:moravf}
 Z := u^2 \d_u + v^2 \d_v
\end{equation}
with null coordinates $u:=t+|x|$ and $v:=t-|x|$ used in the context of
vector field methods to obtain decay. In fact, $Z$ is the pullback by
$\Phi$ of the time-like Killing vector field $\d_t$ on $\T^-$, \[ Z =
(t^2+x^2) \d_t + 2tx\cdot\nabla = \Phi^* \d_t. \]

\xnewpage
\section{Local energy estimate} \label{sec:lee}

Having in mind the transformed wave equation~\eqref{eq:ctwav}, the aim of this
section is to establish a local energy estimate for equations of the form
\begin{equation} \label{eq:vwav}
 \Box \psi + c(t,x) \psi |\psi|^{p-1} = 0
\end{equation}
with a non-negative function $c$ of class $C^2$.  An energy density which
is expected to be useful for such equations is \[ e := \frac12 (\d_t \psi)^2 +
\frac12 (\nabla \psi)^2 + \frac1{p+1} c |\psi|^{p+1}. \] While the associated
energy will, in general, not be conserved, it will be sufficient to prove
boundedness of $\psi$ on the relevant region of $\T^-$, provided the
function $c$ is monotonically decreasing in $t$ there.

Let $\psi \in C^2(\T^-)$ be a classical solution of \eqref{eq:vwav} on $\T^-$
for $p>2$.  Then the vector field $\E$ given by \[ \E:= e \d_t - \bigl(
\d_t \psi \, \nabla \psi \bigr)\cdot \grad \] is continuously differentiable and
\begin{align*}
 \Div \E & = \Bigl[ \Box \psi + c \psi |\psi|^{p-1} \Bigr] \d_t \psi
  + \frac1{p+1} (\d_t c) |\psi|^{p+1} \\
 & = \frac1{p+1} (\d_t c) |\psi|^{p+1}
\end{align*}
holds on $\T^-$. Assume furthermore that $\d_t c \leq 0$ is non-positive
then the same is true for $\Div \E$. This implies, recalling the assumption
that the initial data is compactly supported in an open ball of radius
$\alpha=1/2$ about the origin, that
\begin{align*}
 0 & \geq \int_S (\Div \E) \omega = \int_S d(i_\E \omega) = \int_{\d S}
     i_\E \omega = \int_{D_1(-1,0)} i_\E \omega - \int_{\Phi D_\alpha(1,0)}
     i_\E \omega \\
   & = \int_{D_1(-1,0)} (i_\E dt) dx^1 \wedge dx^2 \wedge dx^3 - \int_{\Phi
     D_\alpha(1,0)} i_\E \omega \\
   & = \int_{D_1(-1,0)} \biggl[ \frac12 (\d_t \psi)^2 + \frac12 (\nabla
     \psi)^2 + \frac1{p+1} c |\psi|^{p+1} \biggr] - \int_{D_\alpha(1,0)}
     \Phi^*(i_\E \omega)
\end{align*}
where $i_\E$ denotes the interior multiplication with $\E$, $\omega = dt
\wedge dx^1 \wedge dx^2 \wedge dx^3$ is the standard volume form and all
hypersurfaces shall be oriented by the transverse vector field $\d_t$. The
region $S \subset \T^-$ is defined as this part of the future of the image
$\Phi D_\alpha(1,0)$ which lies in the past of the hypersurface $\{ t =
-1 \}$. Here, $D_\rho(t,x):=\{ (t,y) \fw |y-x|<\rho\}$ is an open disk of
radius $\rho$ around $(t,x)$. As the integral over $D_\alpha(1,0)$ depends
only on the initial data $\phi_0$ and $\phi_1$ there is a constant $C$
such that \[ \biggl| \int_{D_\alpha(1,0)} \Phi^*(i_\E \omega) \biggr| \leq
C.\] Therefore it follows that
\begin{equation} \label{eq:Ebd}
 E_0:= \int_{D_1(-1,0)}  \biggl[ \frac12 (\d_t \psi)^2 + \frac12 (\nabla
 \psi)^2 + \frac1{p+1} c |\psi|^{p+1} \biggr] \leq C.
\end{equation}

The energy $E_0$ now controls past light-cone integrals of
$c|\psi|^{p+1}$ or, more generally, of the flux density \[ i_\E
\biggl[ dt + \frac{y-x}{|y-x|} \cdot \nabla \biggr] = \frac12 \biggl(
\frac{y-x}{|y-x|} \d_t \psi \biggr)^2 + \frac1{p+1} c|\psi|^{p+1} \] at any
point $(t,x) \in \F:=\T^- \cap \{-1 \leq t < 0\}$ according to the following
Proposition~\ref{lci}.  For notational convenience, let $K(t,x):=\{(s,y)
\fw -1\leq s < t,\ |y-x| \leq t-s \}$ be the solid truncated backward
light-cone at $(t,x)$ and $M(t,x):=\{(s,y) \fw -1\leq s < t,\ |y-x| =
t-s \}$ its mantle, where $(t,x)$ is a point in the relevant region $\F$.
\begin{Prop} \label{lci}
 Let $\psi \in C^2(\F)$ be a solution of \eqref{eq:vwav} with initial
 data supported in $D_1(-1,0)$ satisfying the estimate~\eqref{eq:Ebd}. If
 $\d_t c \leq 0$ on $\F$ then \[ \Flux(t,x) := \frac1{\sqrt 2} \int_{M(t,x)}
 \biggl[ \frac12 \biggl( \frac{y-x}{|y-x|} \d_t \psi \biggr)^2 + \frac1{p+1}
 c|\psi|^{p+1} \biggr] \leq E_0\] holds for any $(t,x) \in \F$.
\end{Prop}

\begin{proof}
Fix $(t,x) \in \F$. Since $\E$ is continuously differentiable, direct
application of Stokes's theorem yields
\begin{align*}
 \int_{K(t,x)} \Div \E & = \int_{M(t,x)} i_\E \omega - \int_{D_{1+t}(-1,x)}
  i_\E \omega \\
 & = \frac1{\sqrt 2} \int_{M(t,x)} i_\E \biggl[ dt + \frac{y-x}{|y-x|}
  \cdot \nabla \biggr] - \int_{D_{1+t}(-1,x)} i_\E dt \\
 & = \frac1{\sqrt 2} \int_{M(t,x)} \biggl[ \frac12 \biggl( \frac{y-x}{|y-x|}
  \d_t \psi \biggr)^2 + \frac1{p+1} c|\psi|^{p+1} \biggr] -
  \int_{D_{1+t}(-1,x)} e \\
 & = \Flux(t,x) - \int_{D_{1+t}(-1,x)} e.
\end{align*}
Hence, it follows that
\begin{align*}
 \Flux(t,x) & = \int_{D_{1+t}(-1,x)} e + \int_{K(t,x)} \Div \E \\
 & \leq \int_{D_1(-1,0)} \biggl[ \frac12 (\d_t \psi)^2 + \frac12 (\nabla
  \psi)^2 + \frac1{p+1} c |\psi|^{p+1} \biggr] + \int_{K(t,x)} \Div \E \\
 & \leq E_0.
\end{align*}
\end{proof}

\xnewpage
\section{Boundedness and weak decay}

Using the local energy estimate established in Proposition~\ref{lci} the
boundedness of $\psi$ in the region $\F=\T^- \cap \{-1 \leq t < 0\}$ is
quite immediate. The argument is exactly the same as in \cite{BSzpak-Defocus}
and goes back to Pecher \cite{Pecher}.

\begin{Theorem} \label{bound}
 Let $\psi \in C^2(\F)$ be a solution of \eqref{eq:vwav} with $2 <
 p < 5$ and initial data supported in $D_1(-1,0)$ satisfying the
 estimate~\eqref{eq:Ebd}. If $c \geq 0$ is uniformly bounded and $\d_t c
 \leq 0$ on $\F$ then $\psi$ is uniformly bounded.
\end{Theorem}

\begin{proof}
Let $\psi_0 \in C^2(\F)$ be a classical solution of the homogeneous
equation $\Box \psi_0 = 0$ with the same initial data as $\psi$. Then,
for a fixed $(t,x) \in \F$, it holds that \[ |\psi-\psi_0|(t,x)
\leq \frac1{4\pi} \int_{-1}^t \int_{\partial B_{t-s}(x)} \frac{c(s,y)
|\psi|^p(s,y)}{t-s}\,dy\,ds. \] Due to the fact that $2 < p < 5$ there
exists a $q$ with $3/2 < q < (p+1)/(p-1)$. Changing variables and applying
H\"older's inequality yields
\begin{equation} \label{eq:hoelder}
\begin{split}
 \lefteqn{ \int_{-1}^t \int_{\partial B_{t-s}(x)} \frac{c
  |\psi|^p(s,y)}{t-s}\,dy\,ds} \quad \\
 & = \int_{B_{1+t}(x)} \frac{c |\psi|^p(t-|y-x|,y)}{|y-x|}\,dy \\
 & \leq \biggl[ \int_{B_{1+t}(x)} c^q |\psi|^{pq}(t-|y-x|,y) \,dy
  \biggr]^\frac1q \biggl[ \int_{B_{1+t}(x)} |y-x|^{-\frac q{q-1}} \,dy
  \biggr]^\frac{q-1}q.
\end{split}
\end{equation}
Consider the first integral. Since $q>1$ and $c$ is bounded, so is
$c^{q-1}$. Furthermore, $0<pq-(p+1)<q$, so that
\begin{align*}
 \lefteqn{ \int_{B_{1+t}(x)} c^q |\psi|^{pq}(t-|y-x|,y) \,dy } \quad \\
 & \leq C \|\psi\|^{pq-(p+1)}_{L^\infty(M(t,x))} \int_{B_{1+t}(x)} c
  |\psi|^{p+1}(t-|y-x|,y) \,dy \\
 & \leq C \biggl( \sup_{-1\leq \tau\leq t} \|\psi(\tau,\cdot)\|_{L^\infty}
  \biggr)^{\gamma q} \frac1{\sqrt 2} \int_{M(t,x)} c |\psi|^{p+1},
\end{align*}
where $0<\gamma<1$ is such that $\gamma q = pq-(p+1)$.  By virtue of
Proposition~\ref{lci} \[ \frac1{\sqrt 2} \int_{M(t,x)} c |\psi|^{p+1} \leq
(p+1) \Flux(t,x) \leq C E_0. \] The second integral in \eqref{eq:hoelder}
can be calculated directly to give \[ \int_{B_{1+t}(x)} |y-x|^{-\frac q{q-1}}
\,dy = 4\pi \int_{0}^{1+t} r^\frac{q-2}{q-1}\,dr = 4\pi \frac{q-1}{2q-3}
(1+t)^\frac{q-1}{2q-3} \leq C \] because $q>3/2$. To sum up, the estimate
\begin{align*}
 |\psi-\psi_0|(t,x) & \leq C E_0^\frac1q \biggl( \sup_{-1\leq \tau\leq t}
  \|\psi(\tau,\cdot)\|_{L^\infty} \biggr)^\gamma
\end{align*}
holds true for any $(t,x) \in \F$. But since the solution $\psi_0$ of the
homogeneous equation is clearly bounded and $0 < \gamma < 1$ this estimate
implies the boundedness of $\psi$ itself uniformly on $\F$.
\end{proof}

As a short digression, note that the proof of Theorem~\ref{bound}
given above uses the boundedness of light-cone integrals of the quantity
$c|\psi|^{p+1}$ which is controlled by the ``potential'' part of the flux
defined in Proposition~\ref{lci}. Alternatively, following Shatah and Struwe
\cite{Shatah-Struwe}, the integral in \eqref{eq:hoelder} could have been
estimated by
\begin{equation} \label{eq:althoelder}
\begin{split}
 \lefteqn{ \int_{B_{1+t}(x)} \frac{c |\psi|^p(t-|y-x|,y)}{|y-x|}\,dy } \quad \\
 & \leq \biggl[ \int_{B_{1+t}(0)} c^2 |\psi|^{2(p-1)}(t-|y|,x+y) \,dy
  \biggr]^\frac12 \biggl[ \int_{B_{1+t}(0)} \frac{|\psi|^2(t-|y|,x+y)}{|y|^2}
  \,dy \biggr]^\frac12 .
\end{split}
\end{equation}
The first integral would then be controlled by \[(1+t)^{4-p}
\|\psi\|_{L^6(M(t,x))}^{2(p-1)}\] when $1<p\leq 4$ and by
\[\|\psi\|_{L^\infty(M(t,x))}^{2(p-4)} \|\psi\|_{L^6(M(t,x))}^6 \]
when $4<p<5$.  But now, contrarily, the ``kinetic'' part of the flux
could be used to estimate both the $L^6$-norm of $\psi$ on $M(t,x)$ by
Sobolev embedding as well as the second integral in \eqref{eq:althoelder}
by Hardy's inequality.

With the function $\psi$ bounded on $\F$ a decay of $\phi$ towards the
future follows directly.

\begin{Cor} \label{decay}
 Let $\phi$ be a classical solution of the wave equation~\eqref{eq:wav}
 for $3\leq p < 5$ with initial data $\phi_0 \in C^3(\R^3)$ and $\phi_1
 \in C^2(\R^3)$ given at $t=1$ and which exists globally towards the
 future. Assume that the support of $\phi_0$ and $\phi_1$ is contained within
 the open ball $B_\alpha(0)$ of radius $\alpha$ about the origin. Then there
 is a constant $C>0$ such that $\phi$ satisfies the decay estimate
 \begin{equation}\label{eq:main-estimate-Cor}
   |\phi(t,x)| \leq \frac C{(1+t+|x|)(1+t-|x|)}
 \end{equation}
 for all $t\geq 1$ and $x \in \R^3$.
\end{Cor}

\begin{proof}
Given such a solution $\phi$ of class $C^2$, it was shown in
Section~\ref{sec:ct} that its conformal transformation $\psi =
\Phi_* (\Omega^{-1} \phi)$ is a classical solution of the wave
equation~\eqref{eq:vwav} on the future of $\Phi D_\alpha(1,0)$ in
$\T^-$ with \[ c = \Phi_* \Omega^{p-3} = (t^2-x^2)^{p-3} \] according to
equation~\eqref{eq:ctwav}.  So $c$ is certainly bounded on the future of
$\Phi D_\alpha(1,0)$ in $\T^-$ because $p\geq 3$ and \[ \d_t \bigl[ \Phi_*
\Omega^{p-3} \bigr] = 2 (p-3) t \Phi_* \Omega^{p-4}, \] which implies $\d_t
c \leq 0$ on $\T^-$ again by reason of $p\geq 3$. Moreover, as detailed
in Section~\ref{sec:lee}, the assumptions on the support of $\phi_0$ and
$\phi_1$ guarantee that the support of the transformed solution restricted
to $\{t=-1\}$ is compactly contained in $D_1(-1,0)$ and that therefore the
estimate~\eqref{eq:Ebd} holds. Thus, Theorem~\ref{bound} applies and yields
boundedness of $\psi$ on $\F$. Since $\psi$ is also bounded on the compact
region $S$ it follows that $\psi$ is bounded on the whole future of $\Phi
D_\alpha(1,0)$ in $\T^-$, say $|\psi| \leq C$. But then \[ |\phi(t,x)| =
\Omega(t,x) |\Phi^* \psi(t,x)| \leq C \Omega(t,x) \] on the whole future of
$D_\alpha(1,0)$ in $\T^+$.  In this region, $t+|x|\geq 1$ and $t-|x| \geq
1-\alpha$, so that there \[ \Omega(t,x) = \frac1{t^2-x^2} \leq \frac{2\Bigl(
1+\frac1{1-\alpha} \Bigr)}{(1+t+|x|)(1+t-|x|)}. \] Hence, the bound of
$|\phi|$ can be written in the regularized form \eqref{eq:main-estimate-Cor}.
\end{proof}

\xnewpage
\section{Improvement of the decay estimate}

The pointwise decay result of Corollary \ref{decay},
\begin{equation}
   |\phi(t,x)| \leq \frac C{(1+t+|x|)(1+t-|x|)},
\end{equation}
can be improved by applying this bound to the nonlinear term in the wave equation \eqref{eq:wave-p} and solving the wave equation by inverting the wave operator, i.e. $\phi=\Box^{-1} (-|\phi|^{p-1}\phi) + \chi_{\phi_0,\phi_1}$. Here, $\chi_{\phi_0,\phi_1}$ represents the contribution from the initial data \eqref{eq:initdata}, i.e. solves the linear wave equation $\Box \chi=0$ with $(\chi,\d_t\chi)|_{t=1}=(\phi_0,\phi_1)$ and is well-known to decay as $|\chi(t,x)|\leq C/t$. Due to the Huygens principle in three dimensions $\chi(t,x)$ is supported in the outgoing light-cone $1-\alpha \leq t-|x|\leq 1+\alpha$. Hence, its decay can be written as
\begin{equation}
   |\chi(t,x)| \leq \frac C{(1+t+|x|)(1+t-|x|)^q}
\end{equation}
with any power $q$ and some $C$ depending on $q$.
Since $\phi\in C^2$ is a classical solution and $\Box^{-1}$ is a positive measure on $\R\times\R^3$ we can estimate
\begin{equation}
  |\phi(t,x)| \leq \Box^{-1} |\phi|^p + |\chi(t,x)|
  \leq \Box^{-1} \frac{C}{(1+t+|x|)^p(1+t-|x|)^{p}} + \frac C{(1+t+|x|)(1+t-|x|)^q}.
\end{equation}
The inverse wave operator can be represented by the Duhamel integral formula and bounded pointwise. According to the Lemma 1 from \cite{NS-DecayLemma} we get for $p>2$
\begin{equation}
  \Box^{-1} \frac{1}{(1+t+|x|)^p(1+t-|x|)^{p}} \leq \frac C{(1+t+|x|)(1+t-|x|)^{p-2}}.
\end{equation}
Choosing $q=p-2$ we arrive at our main result
\begin{Cor}
 Under the assumptions of Corollary \ref{decay} there is a constant $C>0$ such that $\phi$ satisfies the improved decay estimate
 \begin{equation}
   |\phi(t,x)| \leq \frac C{(1+t+|x|)(1+t-|x|)^{p-2}}
 \end{equation}
 for all $t\geq 1$ and $x \in \R^3$.
\end{Cor}

\xnewpage
\section{Outlook}

In some sense, a very similar problem, the linear wave equation with a strong positive potential
$$ \d_t^2 \phi - \Delta \phi + V(x) \phi = 0 $$
having prescribed decay at spatial infinity $V(x)\sim 1/|x|^k$ still lacks a sharp pointwise decay estimate. It also has a positive definite energy and can be conformally transformed to a form analogous to \eqref{eq:vwav}. However, the function $c(t,x)$ is no more regular at $t=|x|=0$. Nevertheless, we expect that our method can be extended to cover this weakly singular case, too.

\section*{Acknowledgments}
The authors want to express their gratitude to the Mathematical Institute in Oberwolfach and to the organizers of the workshop ``Mathematical Aspects of General Relativity'' (Fall 2009) for hospitality and creative atmosphere during the work on that project. One of the authors (NS) also wants to thank the organizers of the workshop ``Quantitative Studies of Nonlinear Wave Phenomena'' at the Erwin Schr{\"o}dinger International Institute for Mathematical Physics in Vienna (Winter 2010) during which this work has been finished.

\bibliography{biblio}

\begin{thebibliography}{SBCR07}

\bibitem[Asa86]{Asakura}
F.~Asakura.
\newblock Existence of a global solution to a semi-linear wave equation with
  slowly decreasing initial data in three space dimenstions.
\newblock {\em Comm. Part. Diff. Eq.}, 13(11):1459--1487, 1986.

\bibitem[BC08]{PB+TCh-private}
P.~Bizo{\'n} and T.~Chmaj.
\newblock 2008.
\newblock private communications.

\bibitem[BS09]{BSzpak-Defocus}
R.~Bieli and N.~Szpak.
\newblock Global pointwise decay estimates for defocusing radial nonlinear wave
  equations.
\newblock 2009.
\newblock arXiv: 0903.0799 [math.AP].

\bibitem[BSZ90]{Baez-Segal-Zhou}
J.~C. Baez, I.~E. Segal, and Z.-F. Zhou.
\newblock The global {G}oursat problem and scattering for nonlinear wave
  equations.
\newblock {\em J. Func. Anal.}, 93:239--269, 1990.

\bibitem[CBPS83]{ChoqBruh-P-Segal}
Y.~Choquet-Bruhat, S.~M. Paneitz, and I.~E. Segal.
\newblock The {Y}ang-{M}ills equations on the universal cosmos.
\newblock {\em J. Func. Anal.}, 53:112--150, 1983.

\bibitem[Chr86]{Christodoulou-Quasilin}
D.~Christodoulou.
\newblock Global solutions of nonlinear hyperbolic equations for small initial
  data.
\newblock {\em Comm. Pure Appl. Math.}, 39:267--282, 1986.

\bibitem[J{\"o}r61]{Joergens}
K.~J{\"o}rgens.
\newblock Das {A}nfangswertproblem im {G}rossen f{\"u}r eine {K}lasse
  nichtlinearer {W}ellengleichungen.
\newblock {\em Math. Z.}, 77:295--307, 1961.

\bibitem[Mor62]{Morawetz62}
C.~S. Morawetz.
\newblock The limiting amplitude principle.
\newblock {\em Comm. Pure Appl. Math.}, 15:349--361, 1962.

\bibitem[Pec74]{Pecher}
H.~Pecher.
\newblock Das verhalten globaler {L}{\"o}sungen nichtlinearer
  {W}ellengleichungen f{\"u}r gro{\ss}e {Z}eiten.
\newblock {\em Math. Z.}, 136:67--92, 1974.

\bibitem[SBCR07]{NS-PB_Tails}
N.~Szpak, P.~Bizo{\'n}, T.~Chmaj, and A.~Rostworowski.
\newblock Linear and nonlinear tails {II}: spherical symmetry.
\newblock 2007.
\newblock arXiv: math-ph/0712.0493; accepted for publication in Journal of
  Hyperbolic Differential Equations (JHDE).

\bibitem[SS98]{Shatah-Struwe}
J.~Shatah and M.~Struwe.
\newblock {\em Geometric wave equations}, volume~2 of {\em Courant Lecture
  Notes in Mathematics}.
\newblock New York University Courant Institute of Mathematical Sciences, New
  York, 1998.

\bibitem[ST97]{Strauss-T}
W.~Strauss and K.~Tsutaya.
\newblock Existence and blow up of small amplitude nonlinear waves with a
  negative potential.
\newblock {\em Discr. Cont. Dynamical Systems}, 3(2):175--188, 1997.

\bibitem[Str68]{Strauss_SemilinDecay}
W.~Strauss.
\newblock Decay and asymptotics for {$\Box u = F(u)$}.
\newblock {\em J. Func. Anal.}, 2(4):409--457, 1968.

\bibitem[Szp07]{NS-DecayLemma}
N.~Szpak.
\newblock Simple proof of a useful pointwise estimate for the wave equation.
\newblock 2007.
\newblock arXiv: math-ph/0708.2801.

\bibitem[Szp08]{NS-Tails}
N.~Szpak.
\newblock Linear and nonlinear tails {I}: general results and perturbation
  theory.
\newblock {\em Journal of Hyperbolic Differential Equations (JHDE)},
  5(4):741--765, 2008.
\newblock arXiv: math-ph/0710.1782.

\end{thebibliography}
\bibliographystyle{alpha}

\end{document}